\documentclass[a4paper]{amsart}
\usepackage{graphicx}
\usepackage{amssymb}
\usepackage{amsmath}
\usepackage{amsthm}
\usepackage{amscd}
\usepackage{color}
\usepackage{colordvi}
\usepackage[all,2cell]{xy}

\UseAllTwocells \SilentMatrices
\newtheorem{thm}{Theorem}[section]

\newtheorem{cor}[thm]{Corollary}
\newtheorem{lem}[thm]{Lemma}
\newtheorem{exm}[thm]{Example}

\newtheorem{prop}[thm]{Proposition}
\theoremstyle{definition}
\newtheorem{defn}[thm]{Definition}
\theoremstyle{remark}
\newtheorem{rem}[thm]{\bf Remark}
\numberwithin{equation}{section}

\begin{document}
\title[Gorenstein homological properties of tensor rings]{Gorenstein homological properties of tensor rings}
\author[Xiao-Wu Chen, Ming Lu] {Xiao-Wu Chen, Ming Lu$^*$}

\thanks{$^*$ The corresponding author}
\thanks{}
\date{\today}

\thanks{E-mail: xwchen$\symbol{64}$mail.ustc.edu.cn; luming$\symbol{64}$scu.edu.cn}

\maketitle

\dedicatory{}%
\commby{}%

\begin{abstract}
Let $R$ be a two-sided noetherian ring and $M$ be a nilpotent $R$-bimodule, which is finitely generated on both sides.  We study  Gorenstein homological properties of the tensor ring $T_R(M)$. Under certain conditions, the ring $R$ is Gorenstein if and only if so is $T_R(M)$. We characterize Gorenstein projective $T_R(M)$-modules in terms of $R$-modules.
\end{abstract}

\section{Introduction}

Let $R$ be a two-sided noetherian ring. In Gorenstein homological algebra, the following Gorenstein homological properties of $R$ are the main concerns: the Gorensteinness of the ring $R$, the (stable) category of Gorenstein projective $R$-modules, and the Gorenstein projective dimensions and resolutions of $R$-modules.

Let $M$ be an $R$-bimodule, which is finitely generated on both sides. The  classical homological properties of the tensor ring $T_R(M)$ are studied  in \cite{Ro, Cohn}. In general, the tensor ring $T_R(M)$ is not noetherian. Hence, we require that $M$ is \emph{nilpotent}, that is, its $n$-th tensor power vanishes for $n$ large enough, in which case $T_R(M)$ is two-sided noetherian.

We are concerned with the Gorenstein homological properties of the tensor ring $T_R(M)$ for a nilpotent $R$-bimodule $M$. The motivation is the example in \cite{GLS}, which is a tensor ring and  is $1$-Gorenstein, that is, the regular module has self-injective dimension at most one on each side. More precisely, the ring $R$ considered in \cite{GLS}  is a certain self-injective algebra and the nilpotent $R$-bimodule $M$ is projective on each side.  Then the tensor ring $T_R(M)$ is $1$-Gorenstein, whose modules yield a characteristic-free categorification of the root system. We mention other related examples in \cite{XZ, Wang, LY}.

The main result of this paper is a vast generalization of the mentioned examples. An $R$-bimodule $M$ is called \emph{perfect} provided that it has finite projective dimension on each side and satisfies the following Tor-vanishing conditions: ${\rm Tor}_j^{R}(M, M^{\otimes_R i})=0$ for each $i, j\geq 1$, where $M^{\otimes_R i}$ is the $i$-th tensor power of $M$.

\vskip 5pt

\noindent {\bf Theorem.}\quad \emph{Let $R$ be a two-sided noetherian ring, and $M$ be a nilpotent perfect $R$-bimodule.  Then the ring $R$ is Gorenstein if and only if so is the tensor ring $T_R(M)$.}

\vskip 5pt

We give an inequality between the Gorenstein dimensions of $R$ and $T_R(M)$. For details, we refer to Theorem \ref{thm:3}.

We characterize Gorenstein projective $T_R(M)$-modules  in terms of $R$-modules; see Theorem \ref{thm:2}. This extends the corresponding description in \cite[Theorem 10.9]{GLS}.  It is well known that the study of Gorenstein projective modules are intimately related to that of Frobenius categories; compare \cite{Bel3, Chen}. The characterization of Gorenstein projective $T_R(M)$-modules relies on an explicit construction of a new Frobenius category; see Theorem \ref{thm:1}. It seems that this construction might be of independent interest.

The paper is structured as follows. In Section 2, we construct a new Frobenius category,  using the category of representations of a nilpotent endofunctor on an abelian category. In Section 3, we characterize Gorenstein projective $T_R(M)$-modules for a certain nilpotent $R$-bimodule $M$. We also study the Gorenstein projective dimensions of $T_R(M)$-modules. In Section 4, we introduce the notion of a perfect bimodule and prove Theorem \ref{thm:3}. Some (non-)examples  are studied in the end.

\section{The construction of a new Frobenius category}

 In this section, we construct a new Frobenius category,  which is an exact subcategory in the category of representations of a certain nilpotent endofunctor on an abelian category. For exact categories, we refer to \cite[Appendix A]{Ke3}.

\subsection{The category of representations}

Let $\mathcal{A}$ be an additive category with an additive endofunctor $F\colon \mathcal{A}\rightarrow \mathcal{A}$. By a \emph{representation} of $F$, we mean a pair $(X, u)$ with $X$ an object and $u\colon F(X)\rightarrow X$ a morphism in $\mathcal{A}$. A morphism $f\colon (X, u)\rightarrow (Y, v)$ between two representations is a morphism $f\colon X\rightarrow Y$  in $\mathcal{A}$ satisfying $f\circ u=v\circ F(f)$. This defines the category ${\rm rep}(F)$ of representations of $F$. We have a forgetful functor
$$U\colon {\rm rep}(F)\longrightarrow \mathcal{A}$$
sending $(X, u)$ to the underlying object $X$.

We assume that $F$ is \emph{nilpotent}, that is $F^{N+1}=0$ for some $N\geq 0$. For each object $A$, we define ${\rm Ind}(A)=\bigoplus_{i=0}^N F^i(A)$ with $F^0={\rm Id}_\mathcal{A}$, and a morphism $c_A\colon F{\rm Ind}(A)\rightarrow {\rm Ind}(A)$ such that its restriction to $F(F^i(A))=F^{i+1}(A)$ is the inclusion into ${\rm Ind}(A)$. This defines a representation $({\rm Ind}(A), c_A)$ of $F$. Moreover, we have the \emph{induction} functor $${\rm Ind}\colon \mathcal{A}\longrightarrow {\rm rep}(F)$$
 sending $A$ to $({\rm Ind}(A), c_A)$, and a morphism $f$ to $\bigoplus_{i=0}^N F^i(f)$.

 \begin{lem}\label{lem:adj}
 Keep the notation as above. Then the pair $({\rm Ind}, U)$ is adjoint.
 \end{lem}

 \begin{proof}
 The natural isomorphism
 $${\rm Hom}_{{\rm rep}(F)}(({\rm Ind}(A), c_A), (X, u))\stackrel{\sim}\longrightarrow {\rm Hom}_\mathcal{A}(A, X)$$
 sends $f$ to the restriction $f|_{A}$. The inverse sends $g\colon A\rightarrow X$ to $g'\colon ({\rm Ind}(A), c_A)\rightarrow (X, u)$, such that the restriction of $g'$ to $F^0(A)=A$ is $g$, and to $F^i(A)$ is given by $u\circ F(u)\circ \cdots \circ F^{i-1}(u)\circ F^i(g)$ for $i\geq 1$.
 \end{proof}

We refer to \cite[Chapter VI]{McL} for monads and monadic adjoint pairs.

\begin{rem}
The nilpotent endofunctor $F$ defines a monad $M$ on $\mathcal{A}$, which, as a functor, equals $\bigoplus_{i=0}^n F^i$ and whose multiplication is induced by the composition of functors. There is an isomorphism of categories between ${\rm rep}(F)$ and $M\mbox{-Mod}_\mathcal{A}$, the category of $M$-modules in $\mathcal{A}$. In other words, the adjoint pair $({\rm Ind}, U)$ is strictly monadic.
\end{rem}

We characterize the essential image of the induction functor.

\begin{lem}\label{lem:ind}
A representation $(X, u)$ is isomorphic to $({\rm Ind}(A), c_A)$ for some object $A$ if and only if $u$ is a split monomorphism which has a cokernel in $\mathcal{A}$.
\end{lem}

\begin{proof}
For the ``only if" part, we just observe that the cokernel of $c_A$ is the natural projection on $A$.

 For the ``if" part, we assume that the cokernel of $u$ is $\pi\colon X\rightarrow A$. Take a section $\iota \colon A\rightarrow X$ of $\pi$. Then $X$ is isomorphic to $A\oplus F(X)$, which is further isomorphic to $A \oplus F(A\oplus F(X))$. Using induction and the nilpotency of $F$, we infer that $X$ is isomorphic to ${\rm Ind}(A)=A\oplus F(A)\oplus \cdots \oplus F^N(A)$; moreover, the corresponding inclusion map $F^i(A)\rightarrow X$ is given by $u \circ F(u)\circ \cdots \circ F^{i-1}(u) \circ F^i(\iota)$. In other words, this gives rise to an isomorphism $({\rm Ind}(A), c_A)\rightarrow (X, u)$, which corresponds to $\iota\colon A\rightarrow X$ in the adjoint pair $({\rm Ind}, U)$; compare the proof of Lemma \ref{lem:adj}.
\end{proof}

We assume now that $\mathcal{A}$ is abelian and that the endofunctor $F$ is right exact. Then ${\rm rep}(F)$ is an abelian category. Moreover, a sequence $(X, u)\stackrel{f}\rightarrow (Y, v)\stackrel{g}\rightarrow (Z, w)$ is exact in ${\rm rep}(F)$ if and only if the underlying sequence $X\stackrel{f}\rightarrow Y\stackrel{g}\rightarrow Z$ is exact in $\mathcal{A}$. In particular, the forgetful functor $U\colon {\rm rep}(F)\rightarrow \mathcal{A}$ is exact.

For each $(X, u)\in {\rm rep}(F)$, there is an exact sequence in ${\rm rep}(F)$
\begin{align}\label{equ:exact}
0\longrightarrow ({\rm Ind}(FX), c_{FX}) \xrightarrow{\phi_{(X, u)}} ({\rm Ind}(X), c_X) \xrightarrow{\eta_{(X, u)}}  (X, u)\longrightarrow 0.
\end{align}
Here, we view $\phi_{(X, u)}$ as a formal matrix; the only possibly nonzero entries are ${\rm Id}_{F^i(X)}\colon F^i(X)\rightarrow F^i(X)$ and $-F^{i-1}(u)\colon F^i(X)\rightarrow F^{i-1}(X)$ for $i\geq 1$. The morphism $\eta_{(X, u)}$ is given by the counit of the adjoint pair $({\rm Ind}, U)$. More precisely, the restriction of $\eta_{(X, u)}$  on $X$ is ${\rm Id}_X$, and on $F^i(X)$ is $u\circ F(u)\circ \cdots \circ F^{i-1}(u)$ for $i\geq 1$. The above  exact sequence is a categorical version of the one in \cite[Lemma]{Ro}.

\subsection{A new Frobenius category}

Let $\mathcal{A}$ be an abelian category. A full additive subcategory $\mathcal{E}$ is \emph{exact} provided that it is closed under extensions. In this case, $\mathcal{E}$ becomes naturally an exact category in the sense of Quillen, whose conflations are given by short exact sequences with terms in $\mathcal{E}$.

Recall that an exact category $\mathcal{E}$ is \emph{Frobenius} provided that it has enough projective objects and enough injective objects, and the class of projective objects coincides with the class of injective objects. Frobenius categories are important, since their stable categories modulo projective objects have  natural triangulated structures; see \cite[Section I.2]{Hap}.

We will consider the following conditions for the triple $(\mathcal{A}, \mathcal{A}', F)$.
\begin{enumerate}
\item[(F1)] The category $\mathcal{A}$ is abelian, and $\mathcal{A}'\subseteq \mathcal{A}$ is an exact subcategory, which is Frobenius as an exact category. Denote by $\mathcal{P}\subseteq \mathcal{A}'$ the subcategory of projective objects.
\item[(F2)] For every epimorphism $f\colon A\rightarrow X$ with $X\in \mathcal{A}'$, there is an epimorphism $g\colon Y\rightarrow A$ with $Y\in \mathcal{A}'$ and ${\rm Ker}(f\circ g)\in \mathcal{A}'$.
\item[(F3)] The endofunctor $F\colon \mathcal{A}\rightarrow \mathcal{A}$ is right exact and nilpotent, satisfying that ${\rm Ext}^1_\mathcal{A}(X, F^i(P))=0={\rm Ext}^1_\mathcal{A}(P, F^i(X))$ for any $X\in \mathcal{A}'$, $P\in \mathcal{P}$ and $i\geq 1$.
    \item[(F4)] For any exact sequence $\eta\colon 0\rightarrow A \rightarrow B\rightarrow X\rightarrow 0$ in $\mathcal{A}$, we have that $F(\eta)$ is exact,  provided that $X$ admits a monomorphism $u\colon F(X)\rightarrow X$ with ${\rm Cok}u\in \mathcal{A}'$.
\end{enumerate}

\begin{lem}\label{lem:defn}
Assume that the triple $(\mathcal{A}, \mathcal{A}', F)$ satisfies {\rm (F1)-(F4)}. Then the following statements hold.
\begin{enumerate}
\item For every epimorphism $f\colon A\rightarrow X$ with $X\in \mathcal{A}'$, there is an epimorphism $g'\colon P\rightarrow A$ with $P\in \mathcal{P}$ and ${\rm Ker}(f\circ g')\in \mathcal{A}'$.
\item For an exact sequence $0\rightarrow A\stackrel{f}\rightarrow B\rightarrow X\rightarrow 0$ with $X\in \mathcal{A}'$ and a morphism $a\colon A\rightarrow F^i(P)$ for some $P\in \mathcal{P}$ and $i\geq 0$, there is a morphism $b\colon B\rightarrow F^i(P)$ with $a=b\circ f$.
\item For an exact sequence $\eta\colon 0\rightarrow A\rightarrow B\rightarrow F^i(X)\rightarrow 0$ with $X\in \mathcal{A}'$ and $i\geq 0$, we have that $F(\eta)$ is exact.
    \item For any monomorphism $f\colon A\rightarrow B$ with cokernel in $\mathcal{A}'$, we have that $F^i(f)$ is mono for any $i\geq 1$.
        \item Assume that $u\colon F(X)\rightarrow X$ is a monomorphism with ${\rm Cok}u\in \mathcal{A}'$. Then ${\rm Ext}_\mathcal{A}^1(P, X)=0$ for each $P\in \mathcal{P}'$.
\end{enumerate}
\end{lem}

\begin{proof}
For (1), we just compose the morphism $g$ in (F2) with an epimorphism $P\rightarrow Y$, whose kernel lies in $\mathcal{A}'$. If $i=0$ in (2),  we use the fact that ${\rm Ext}_\mathcal{A}^1(X, P)=0$, since $P$ is also injective in $\mathcal{A}'$. If $i\geq 1$,  we just apply ${\rm Ext}_\mathcal{A}^1(X, F^i(P))=0$ in (F3).

 For (3), we may assume that $i\leq N$.  By adding a trivial direct summand, we may replace $\eta$ by  $\eta'\colon 0\rightarrow A\rightarrow B'\rightarrow {\rm Ind}(X)\rightarrow 0$ with $B'=B\oplus (\oplus_{j\neq i}F^j(X))$. By the exact sequence
$$0\longrightarrow F{\rm Ind}(X)\stackrel{c_X}\longrightarrow {\rm Ind}(X)\longrightarrow X\longrightarrow 0,$$
we might apply (F4) to $\eta'$. The exactness of $F(\eta')$ implies the one for $F(\eta)$.  The statement (4) follows from (3) and by induction.

For the last statement, we apply (4) to obtain that $F^i(u)$ is mono for each $i\geq 1$. The cokernel of $F^i(u)$ is isomorphic to $F^i({\rm Cok}u)$. By the nilpotency of $F$, we infer that $X$ is an iterated extension of the objects $F^i({\rm Cok}u)$ for $0\leq i\leq N$. By ${\rm Ext}^1_\mathcal{A}(P, F^i({\rm Cok}u))=0$ in (F3), we deduce the required statement.
\end{proof}

The following consideration is inspired by \cite{XZ, LZ}. We consider the following full subcategory of ${\rm rep}(F)$.
$$\mathcal{B}=\{(X, u) \in {\rm rep}(F)\; |\; u \mbox{  is a monomorphism with } {\rm Cok}u\in \mathcal{A}'\}.$$

\begin{thm}\label{thm:1}
Assume that the triple $(\mathcal{A}, \mathcal{A}', F)$ satisfies {\rm (F1)-(F4)}. Then $\mathcal{B}\subseteq {\rm rep}(F)$ is an exact subcategory, which is a Frobenius category. Moreover, its projective objects are precisely of the form $({\rm Ind}(P), c_P)$ for $P\in \mathcal{P}$.
\end{thm}

\begin{proof}
\emph{Step 1} \; To show that $\mathcal{B}\subseteq {\rm rep}(F)$ is closed under extensions, we take an exact sequence $0\rightarrow (X, u)\rightarrow (Y, v)\rightarrow (Z, w)\rightarrow 0$ with $(X, u)$ and $(Z, w)$ in $\mathcal{B}$. Since the sequence $0\rightarrow F(X)\rightarrow F(Y)\rightarrow F(Z)\rightarrow 0$ is exact by (F4), we infer by the five lemma that $v$ is a monomorphism, whose cokernel is an extension of ${\rm Cok}w$ by ${\rm Cok}u$. Since $\mathcal{A}'\subseteq \mathcal{A}$ is closed under extensions, the cokernel of $v$ lies in $\mathcal{A}'$. This proves that $(Y, v)$ lies in $\mathcal{B}$.

\emph{Step 2} \;  We claim that $({\rm Ind}(P), c_P)$ is projective in $\mathcal{B}$ for $P\in \mathcal{P}$. Take an exact sequence $\xi\colon 0\rightarrow (X, u) \rightarrow (Y, v) \rightarrow (Z, w)\rightarrow 0$ in $\mathcal{B}$. By the adjoint pair in Lemma~\ref{lem:adj}, the sequence ${\rm Hom}_{{\rm rep}(F)}(({\rm Ind}(P), c_P), \xi)$  is isomorphic to ${\rm Hom}_\mathcal{A}(P, U(\xi))$. The latter is exact by ${\rm Ext}_\mathcal{A}^1(P, X)=0$ in Lemma \ref{lem:defn}(5). This proves the claim.

Let $(X, u)$ be an object in $\mathcal{B}$. Denote by $\pi\colon X\rightarrow {\rm Cok}u$ the projection. Then by Lemma \ref{lem:defn}(1), there is an epimorphism $f\colon P\rightarrow X$ with $P\in \mathcal{P}$ and ${\rm Ker}(\pi\circ f)\in \mathcal{A}'$. By the adjoint pair  in Lemma \ref{lem:adj}, the morphism $f$ corresponds to a morphism $f'\colon ({\rm Ind}(P), c_P)\rightarrow (X, u)$, which is clearly epic. Then we have an exact sequence in ${\rm rep}(F)$
$$0\longrightarrow (Y, v)\longrightarrow ({\rm Ind}(P), c_P)\stackrel{f'}\longrightarrow (X, u)\longrightarrow 0.$$
By (F4), we have that $0\rightarrow F(Y)\rightarrow F{\rm Ind}(P)\rightarrow F(X)\rightarrow 0$ is exact. Using the five lemma, we infer that $v$ is mono, whose cokernel is isomorphic to ${\rm Ker}(\pi\circ f)$ and thus lies in $\mathcal{A}'$. Hence $(Y, v)$ lies in $\mathcal{B}$. The above exact sequence shows that $\mathcal{B}$ has enough projective objects. Moreover, each projective object is a direct summand of $({\rm Ind}(P), c_P)$ for some $P\in \mathcal{P}$. Using Lemma \ref{lem:ind}, we infer that any projective object  is of the form $({\rm Ind}(Q), c_Q)$ for $Q\in \mathcal{P}$.

\emph{Step 3} \; We claim that $({\rm Ind}(P), c_P)$ is injective for each $P\in \mathcal{P}$. For this, we take an arbitrary exact sequence $0\rightarrow (X, u)\stackrel{f}\rightarrow (Y, v)\stackrel{g}\rightarrow (Z, w)\rightarrow 0$ in $\mathcal{B}$. Then we have the following exact commutative diagram
\[\xymatrix{
0\ar[r] & F(X)\ar[d]^-{u}\ar[r]^-{F(f)} & F(Y) \ar[d]^-{v}\ar[r]^-{F(g)} & F(Z)\ar[d]^-{w}\ar[r] & 0\\
0\ar[r] & X\ar[r]^-{f} \ar[d]^{\pi_1} & Y \ar[r]^-{g}\ar[d]^-{\pi_2} & Z\ar[r] \ar[d]^-{\pi_3} & 0\\
0\ar[r] & {\rm Cok}u\ar[r]^-{\bar{f}} & {\rm Cok}v \ar[r]^-{\bar{g}} & {\rm Cok}w\ar[r] & 0.
}\]
Take any morphism
$$(a_0, a_1, \cdots, a_N)^t\colon (X, u)\longrightarrow ({\rm Ind}(P), c_P),$$
 where ``$t$" denotes the transpose and $a_i\colon X\rightarrow F^i(P)$. We observe that $a_0\circ u=0$ and $a_i\circ u=F(a_{i-1})$ for $i\geq 1$. Then there exists a morphism $\bar{a}_0\colon {\rm Cok}u\rightarrow P$ with $a_0=\bar{a}_0\circ \pi_1$. By Lemma \ref{lem:defn}(2) we have a morphism $\bar{b}_0\colon  {\rm Cok}v\rightarrow P$ with $\bar{a}_0=\bar{b}_0\circ \bar{f}$. We set $b_0=\bar{b}_0\circ \pi_2\colon Y\rightarrow P$. Hence, we have $a_0=b_0\circ f$ and $b_0\circ v=0$.

Since ${\rm Cok}v$ lies in $\mathcal{A}'$, by Lemma \ref{lem:defn}(2) we have a morphism $b'_1\colon Y\rightarrow F(P)$ satisfying $F(b_0)=b'_1\circ v$. We have
$$(a_1-b'_1\circ f)\circ u=F(a_0)-b'_1\circ v \circ F(f)= F(a_0)-F(b_0)\circ F(f)=0.$$
 There exists $x\colon {\rm Cok}u\rightarrow F(P)$ satisfying $a_1-b'_1\circ f=x\circ \pi_1$. Applying Lemma \ref{lem:defn}(2) again, we have a morphism $y\colon {\rm Cok}v\rightarrow F(P)$ with $x=y\circ \bar{f}$. Set $b_1=b'_1+y\circ \pi_2$. Then we have $a_1=b_1\circ f$ and $b_1\circ v=F(b_0)$.

We iterate the above argument to construct $b_i\colon Y\rightarrow F^i(P)$ such that $a_i=b_i\circ f$ and $b_i\circ v=F(b_{i-1})$ hold for $i\geq 2$. Then we have the morphism
 $$(b_0, b_1, \cdots, b_N)^t\colon (Y, v)\longrightarrow ({\rm Ind}(P), c_P),$$
 which makes $(a_0, a_1, \cdots, a_N)^t$ factor though $f$, as required.

\emph{Step 4}\; For the final step, we construct for each object $(X, u)$ in $\mathcal{B}$ an exact sequence
$$0\rightarrow (X, u)\longrightarrow({\rm Ind}(P), c_P)\longrightarrow (Y, v)\longrightarrow 0$$
in $\mathcal{B}$ with $P\in \mathcal{P}$. Then we are done with the whole proof.

Denote by $\pi\colon X\rightarrow {\rm Cok}u$ the cokernel of $u$. We observe that $F^i(u)$ is mono by Lemma \ref{lem:defn}(4). In what follows,  we view $F^i(X)$ as a subobject of $X$.

Since $\mathcal{A}'$ is Frobenius, we take a monomorphism $\iota\colon {\rm Cok}u\rightarrow P$ with its cokernel in $\mathcal{A}'$. Set $a_0=\iota\circ \pi$. Then ${\rm Ker}a_0={\rm Im}u$. Since $F^i(\iota)$ is mono by Lemma \ref{lem:defn}(4), we infer that ${\rm Ker}F^i(a_0)={\rm Im}F^i(u)$ for all $i\geq 1$.

By Lemma \ref{lem:defn}(2), we have a morphism $a_1\colon X\rightarrow F(P)$ with $F(a_0)=a_1\circ u$. Then we have $${\rm Ker}a_1\cap {\rm Im}u={\rm Ker} F(a_0)={\rm Im}F(u).$$
 Similarly, we have a morphism $a_2\colon X\rightarrow F^2(P)$ with $F(a_1)=a_2\circ u$. Then we have
\begin{align*}
{\rm Ker}a_2\cap {\rm Im}u \cap {\rm Im}F(u)&={\rm Ker}F(a_1)\cap {\rm Im} F(u)\\
&={\rm Ker}F^2(a_0)\\
&={\rm Im}F^2(u).
\end{align*}
Here, the first and second equality follow from $F(a_1)=a_2\circ u$ and $F(a_0)=a_1\circ u$, respectively. We proceed in the same way to construct $a_i\colon X\rightarrow F^i(P)$. We observe that
$${\rm Ker}a_N\cap \cdots \cap {\rm Ker}a_1\cap {\rm Ker}a_0={\rm Im} F^N(u)=0.$$
This gives rise to a monomorphism
$$(a_0, a_1, \cdots, a_N)^t\colon (X, u)\longrightarrow ({\rm Ind}(P), c_P),$$
which induces $\iota$ by taking the cokernels. Denote its cokernel  by $(Y, v)$. By the snake lemma and the injectivity of $\iota$, we infer that $v$ is mono, whose cokernel coincides with the one of $\iota$ and thus lies in $\mathcal{A}'$. This proves the desired short exact sequence.
\end{proof}

\section{Gorenstein projective modules and admissible bimodules}
In this section, we study Gorenstein projective modules over a tensor ring. The reference for Gorenstein homological algebra is \cite{EJ}.

Throughout $R$ is a two-sided noetherian ring. We denote by $R\mbox{-mod}$ the abelian category of finitely generated left $R$-modules, and by $R\mbox{-proj}$ the full subcategory of projective modules. We identify right $R$-modules with left $R^{\rm op}$-modules, where $R^{\rm op}$ is  the opposite ring. Hence, $R^{\rm op}\mbox{-mod}$ denotes the category of finitely generated right $R$-modules.

An unbounded complex $P^\bullet$ of projective $R$-modules is \emph{totally acyclic} provided that it is acyclic and that its dual $(P^\bullet)^*={\rm Hom}_R(P^\bullet, R)$ is also acyclic. Recall that an $R$-module $G$ is \emph{Gorenstein projective} provided that there exists a totally acyclic complex $P^\bullet$ with its zeroth cocycle  $Z^0(P^\bullet)\simeq G$. The complex $P^\bullet$ is called a \emph{complete resolution} of $G$.

We denote by $R\mbox{-Gproj}$ the full subcategory of Gorenstein projective modules. We observe that $R\mbox{-proj}\subseteq R\mbox{-Gproj}$. It is well known that $R\mbox{-Gproj}\subseteq R\mbox{-mod}$ is closed under extensions. Moreover, as an exact category, it is Frobenius, whose projective objects are precisely  projective $R$-modules; compare \cite[Proposition 3.8]{Bel3}.

The following fact is standard.

\begin{prop}\label{prop:Gproj}
Let $\mathcal{C}$ be an exact subcategory of $R\mbox{-{\rm mod}}$ containing $R\mbox{-{\rm proj}}$. Assume that $\mathcal{C}$ is Frobenius, whose projective objects are precisely projective $R$-modules. Then $\mathcal{C}\subseteq R\mbox{-{\rm Gproj}}$.
\end{prop}

\begin{proof}
For a module $C\in \mathcal{C}$, we construct its complete resolution by gluing the projective resolution and injective resolution of $C$ inside $\mathcal{C}$.
\end{proof}

For a left $R$-module $X$, we denote by ${\rm pd}_RX$, ${\rm id}_RX$ and ${\rm Gpd}_RX$ the projective dimension, injective dimension and Gorenstein projective dimension of $X$, respectively. Recall that ${\rm Gpd}_RX\leq n$ if and only if there is an exact sequence $0\rightarrow G^{-n} \rightarrow \cdots \rightarrow G^{-1}\rightarrow G^0\rightarrow X \rightarrow 0$ with each $G^{-i}\in R\mbox{-Gproj}$.

Recall that a two-sided noetherian ring $R$ is \emph{Gorenstein} provided that ${\rm id}_RR<\infty$ and ${\rm id}_{R^{\rm op}} R<\infty$. In this case, we have ${\rm id}_RR={\rm id}_{R^{\rm op}} R$ by \cite[Lemma A]{Z}. This common value is denoted by ${\rm G.dim}\; R$. If ${\rm G.dim}\; R\leq d$, we call $R$ a \emph{$d$-Gorenstein ring}. For example, $0$-Gorenstein rings are precisely quasi-Frobenius rings.

\begin{lem}\label{lem:goren}
Let $R$ be a two-sided noetherian ring, and $d, d_1, d_2$ be integers. Then the following statements hold.
\begin{enumerate}
\item If $R$ is $d$-Gorenstein, then ${\rm Gpd}_R X\leq d$ and ${\rm Gpd}_{R^{\rm op}} Y\leq d$ for each left $R$-module $X$ and right $R$-module $Y$.
    \item If ${\rm Gpd}_R X\leq d_1$ and ${\rm Gpd}_{R^{\rm op}} Y\leq d_2$ for each left $R$-module $X$ and right $R$-module $Y$, then $R$ is ${\rm min}\{d_1, d_2\}$-Gorenstein.
\end{enumerate}
\end{lem}

\begin{proof}
We refer to \cite[Theorem 12.3.1]{EJ} for a detailed proof. For (2), we just note that ${\rm Ext}_R^i(X, R)=0$ for $i>{\rm Gpd}_R X$. Hence, the assumptions imply that ${\rm id}_RR\leq d_1$ and  ${\rm id}_{R^{\rm op}}R\leq d_2$. Then we are done by using the fact  ${\rm id}_RR={\rm id}_{R^{\rm op}} R$ for any Gorenstein ring $R$.
\end{proof}

Let $M$ be an $R$-bimodule which is finitely generated on both sides. Write $M^{\otimes_R 0}=R$ and $M^{\otimes_R (j+1)}=M\otimes_R (M^{\otimes_R j})$ for $j\geq 0$. We say that $M$ is \emph{nilpotent}, if $M^{\otimes_R (N+1)}=0$ for some $N\geq 0$. This is equivalent to the condition that the  endofunctor $M\otimes_R-$ on $R\mbox{-mod}$ is nilpotent.

Let $M$ be a nilpotent bimodule. We denote by $T_R(M)=\bigoplus_{i=0}^\infty M^{\otimes_R i}$ the tensor ring, which is also two-sided noetherian. There is an isomorphism of categories
\begin{align}\label{equ:iso-cat}
{\rm rep}(M\otimes_R-)\stackrel{\sim}\longrightarrow T_R(M)\mbox{-mod},
\end{align}
 which identifies a  representation $(X, u)$ of $M\otimes_R-$ with  a left $T_R(M)$-module $X$ such that $m.x=u(m\otimes x)$ for $m\in M$ and $x\in X$. In particular, projective $T_R(M)$-modules correspond to the representations $({\rm Ind}(P), c_P)$ for projective $R$-modules $P$, which might also be viewed as the the scalar extension $T_R(M)\otimes_R P$. Indeed, for each left $R$-module $Z$, $({\rm Ind}(Z), c_Z)$ corresponds to $T_R(M)\otimes_R Z$.

  In what follows, we will identify the two categories in (\ref{equ:iso-cat}). The following results are standard.

 \begin{lem}\label{lem:pd}
 Keep the assumptions as above. Let $Z$ be a left $R$-module. Then the following statements hold.
 \begin{enumerate}
 \item The $R$-module $Z$ is projective if and only if the $T_R(M)$-module $T_R(M)\otimes_R Z$ is projective.
 \item If ${\rm Tor}_i^R(T_R(M), Z)=0$ for each $i\geq 1$, then ${\rm pd}_R Z={\rm pd}_{T_R(M)}T_R(M)\otimes_R Z$.
 \end{enumerate}
 \end{lem}

\begin{proof}
For (1), it suffices to show the ``if" part. Recall the identification of  $({\rm Ind}(Z), c_Z)$ and $T_R(M)\otimes_R Z$. Assume that $T_R(M)\otimes_R Z$ is projective. Hence, $({\rm Ind}(Z), c_Z)$ is a direct summand of $({\rm Ind}(R^n), c_{R^n})$ for some $n\geq 1$.  It follows that $Z$ is projective, since it is isomorphic to the cokernel of $c_Z$.

For (2),  we take an exact sequence
$$0 \rightarrow Y\rightarrow P^{-n}\rightarrow P^{1-n}\rightarrow \cdots \rightarrow P^{-1}\rightarrow P^0\rightarrow Z\rightarrow 0$$
of $R$-modules with each $P^{-i}$ projective. Applying $T_R(M)\otimes_R-$ to it, we get an exact sequence starting at $T_R(M)\otimes_R Y$, with middle terms projective $T_R(M)$-modules.  We apply (1) to $Y$. Then we are done.
\end{proof}

\begin{lem}\label{lem:dual}
We have a natural isomorphism $({\rm Ind}(P), c_P)^* \simeq ({\rm Ind}(P^*), c_{P^*})$ of $T_R(M)^{\rm op}$-modules for each $P\in R\mbox{-{\rm proj}}$.
\end{lem}

\begin{proof}
We are done by the following canonical isomorphisms
\begin{align*}
({\rm Ind}(P), c_P)^* &={\rm Hom}_{T_R(M)} (T_R(M)\otimes_R P, T_R(M))\\
                     &\simeq {\rm Hom}_R(P, T_R(M))\\
                     & \simeq {\rm Hom}_R(P, R)\otimes_R T_R(M)\\
                     & \simeq ({\rm Ind}(P^*), c_{P^*}).
\end{align*}
As mentioned above,  we identify $({\rm Ind}(P), c_P)$ with $T_R(M)\otimes_R P$, and $({\rm Ind}(P^*), c_{P^*})$ with $P^*\otimes_R T_R(M)$.
\end{proof}

\begin{defn}
An $R$-bimodule $M$ is \emph{left-admissible} provided that $${\rm Ext}_R^1(G, M^{\otimes_R i})=0=
{\rm Tor}_1^R(M, M^{\otimes_R i}\otimes_R G)$$
  for each $G\in R\mbox{-{\rm Gproj}}$ and $i\geq 0$. Dually, it is \emph{right-admissible} if $M$ is left-admissible replacing $R$ by its opposite. We say that $M$ is \emph{admissible} if it is both left and right-admissible.\hfill $\square$
\end{defn}

\begin{lem}\label{lem:Gproj}
Let $M$ be an $R$-bimodule. Then the triple $(R\mbox{-{\rm mod}}, R\mbox{-{\rm Gproj}}, M\otimes_R-)$ satisfies {\rm (F1)-(F4)} if and only if $M$ is nilpotent and left-admissible.
\end{lem}

\begin{proof}
For the ``only if" part, it suffices to infer ${\rm Tor}_1^R(M, (M^{\otimes_R i})\otimes_R G)=0$ from Lemma \ref{lem:defn}(3).

For the ``if" part, we have that (F2) holds, since $R\mbox{-Gproj}$ is closed under kernels of epimorphisms. It remains to verify (F4). Assume that an $R$-module $X$ fits into a short exact sequence $\eta\colon 0\rightarrow M\otimes_R X\rightarrow X\rightarrow G\rightarrow 0$ with $G$ Gorenstein projective. By assumption, we infer that $(M^{\otimes_R i})\otimes_R \eta$ is exact for each $i$. Recall that $M$ is nilpotent. It follows that $X$ is an iterated extension of the  modules $(M^{\otimes_R i})\otimes_R G$. By the Tor-vanishing assumption, we infer that ${\rm Tor}^R_1(M, X)=0$. This proves (F4).
\end{proof}

We introduce the following full subcategories of $T_R(M)\mbox{-mod}$ and $T_R(M)^{\rm op}\mbox{-mod}$, respectively
$${\rm Gmon}(M\otimes_R -):=\{(X, u)\in {\rm rep}(M\otimes_R -)\; |\; u \mbox{ is mono with } {\rm Cok}u\in R\mbox{-{\rm Gproj}} \}$$
and
$${\rm Gmon}(-\otimes_R M):=\{(Y, v)\in {\rm rep}(-\otimes_R M)\; |\; v \mbox{ is mono with } {\rm Cok}v\in R^{\rm op}\mbox{-{\rm Gproj}} \}.$$

\begin{lem}\label{lem:incl}
Let $R$ be a two-sided noetherian ring. Assume that $M$ is a nilpotent left-admissible $R$-bimodule. Then we have ${\rm Gmon}(M\otimes_R -) \subseteq T_R(M)\mbox{-{\rm Gproj}}$.
\end{lem}

\begin{proof}
Recall the isomorphism (\ref{equ:iso-cat}), where projective $T_R(M)$-modules are identified with $({\rm Ind}(P), c_P)$ for projective $R$-modules $P$.  By Theorem \ref{thm:1} and Lemma \ref{lem:Gproj}, the category ${\rm Gmon}(M\otimes_R -)$ is Frobenius. Then the required inclusion follows from Proposition \ref{prop:Gproj}.
\end{proof}

\begin{prop}\label{prop:incl}
Let $R$ be a two-sided noetherian ring, and $M$ be a nilpotent $R$-bimodule. Then the following two statements are equivalent:
 \begin{enumerate}
 \item $T_R(M)\mbox{-{\rm Gproj}}\subseteq {\rm Gmon}(M\otimes_R -)$ and $T_R(M)^{\rm op}\mbox{-{\rm Gproj}}\subseteq {\rm Gmon}(-\otimes_R M)$;
 \item for any totally acyclic complex $Q^\bullet$ of $T_R(M)$-modules and its dual $(Q^\bullet)^*={\rm Hom}_{T_R(M)}(Q^\bullet, T_R(M))$, the complexes $M\otimes_R Q^\bullet$ and $(Q^\bullet)^*\otimes_R M$ are both acyclic.
     \end{enumerate}
\end{prop}

\begin{proof}
For ``$(1)\Rightarrow (2)$", we take a totally acyclic complex  $Q^\bullet$ of $T_R(M)$-modules. We assume that $Q^i=({\rm Ind}(P^i), c_{P^i})$ for $P^i\in R\mbox{-proj}$. The $i$-th cocycle of $Q^\bullet$ is denoted by $(Z^i, u^i)$, which is Gorenstein projective. In particular, $u^i$ is injective by the assumption. It follows that the upper row of the following commutative diagram is exact.
\[
\xymatrix{
0\ar[r] & M\otimes_R Z^i \ar[d]_{u^i} \ar[r] & M\otimes_R {\rm Ind}(P^i) \ar[d]^-{c_{P^i}}\ar[r] & M\otimes_R Z^{i+1} \ar[r] \ar[d]^-{u^{i+1}} & 0\\
0\ar[r] & Z^i \ar[r] & {\rm Ind}(P^i) \ar[r] & Z^{i+1} \ar[r] &0
}
\]
Then we infer that $M\otimes_R Q^\bullet$ is acyclic. Since $(Q^\bullet)^*$ is also totally acyclic, the same argument proves that $(Q^\bullet)^*\otimes_R M$ is acyclic.

To prove ``$(2)\Rightarrow (1)$", we only prove the first inclusion. Take $(X, u)\in T_R(M)\mbox{-Gproj}$. Assume that its complete resolution is $Q^\bullet$, where $Q^i=({\rm Ind}(P^i), c_{P^i})$  for projective $R$-modules $P^i$. This gives rise to a commutative exact diagram.
\[
\xymatrix{\cdots \ar[r] & M\otimes_R {\rm Ind}(P^{-1}) \ar[d]^-{c_{P^{-1}}}\ar[r] & M\otimes_R{\rm Ind}(P^{-0}) \ar[d]^-{c_{P^0}} \ar[r] & M\otimes_R{\rm Ind}(P^{1}) \ar[d]^-{c_{P^1}}\ar[r] & \cdots\\
\cdots \ar[r] & {\rm Ind}P^{-1} \ar[d] \ar[r] & {\rm Ind}(P^{-0}) \ar[d] \ar[r] & {\rm Ind}(P^{1}) \ar[d] \ar[r] & \cdots\\
\cdots \ar[r] & P^{-1} \ar[r] & P^{-0} \ar[r] & P^{1} \ar[r] & \cdots
}\]
It follows that $u$ is mono and  the bottom row $P^\bullet$ is acyclic, whose zeroth cocycle is isomorphic to ${\rm Cok}u$.

We consider the dual $(Q^\bullet)^*$, whose components are isomorphic to $({\rm Ind}(P^i)^*, c_{{P^i}^*})$ by Lemma \ref{lem:dual}. Applying the same argument to $(Q^\bullet)^*$, we obtain a commutative exact diagram in $R^{\rm op}\mbox{-mod}$, whose bottom row is exact. But, the bottom  row is isomorphic to $(P^\bullet)^*$, proving the totally acyclicity of $P^\bullet$. Consequently, ${\rm Cok}u$ is Gorenstein projective. It follows that $(X, u)\in {\rm Gmon}(M\otimes_R -)$, as required.
\end{proof}

We summarize the results in this section. Recall that a two-sided noetherian ring $R$ is \emph{CM-free} provided that $R\mbox{-Gproj}=R\mbox{-proj}$. We mention that Gorenstein projective modules are also called maximal Cohen-Macaulay modules. Here, CM stands for Cohen-Macaulay.

\begin{thm}\label{thm:2}
Let $R$ be a two-sided noetherian ring with $M$ a nilpotent admissible $R$-bimodule satisfying the following condition: for any totally acyclic complex $Q^\bullet$ of $T_R(M)$-modules, the complexes $M\otimes_R Q^\bullet$ and ${\rm Hom}_{T_R(M)}(Q^\bullet, T_R(M))\otimes_R M$ are both acyclic. Then we have
$$T_R(M)\mbox{-{\rm Gproj}}= {\rm Gmon}(M\otimes_R -) \mbox{  and  } T_R(M)^{\rm op}\mbox{-{\rm Gproj}}= {\rm Gmon}(-\otimes_R M).$$
Consequently, for any left $R$-module $Z$, the following statements hold.
\begin{enumerate}
\item The module $Z$ lies in $R\mbox{-{\rm Gproj}}$ if and only if $T_R(M)\otimes_R Z$ lies in $T_R(M)\mbox{-{\rm Gproj}}$.
\item If   ${\rm Tor}_i^R(T_R(M), Z)=0$ for each $i\geq 1$, then we have  $${\rm Gpd}_R Z={\rm Gpd}_{T_R(M)} T_R(M)\otimes_R Z.$$
    \item The ring $R$ is CM-free if and only if so is $T_R(M)$.
\end{enumerate}
\end{thm}

\begin{proof}
We just prove the consequences. Recall the isomorphism (\ref{equ:iso-cat}), which identifies $T_R(M)\otimes_R Z$ with the induced representation $({\rm Ind}(Z), c_Z)$. The cokernel of $c_Z$ is isomorphic to $Z$. Then (1) follows immediately. For (2), we take an exact sequence
$$0 \rightarrow Y\rightarrow P^{-n}\rightarrow P^{1-n}\rightarrow \cdots \rightarrow P^{-1}\rightarrow P^0\rightarrow Z\rightarrow 0$$
of $R$-modules with each $P^{-i}$ projective. Applying $T_R(M)\otimes_R-$ to it, we get an exact sequence starting at $T_R(M)\otimes_R Y$, with middle terms projective $T_R(M)$-modules. Recall that ${\rm Gpd}_R Z\leq n+1$ if and only if $Y$ is Gorenstein projective. A similar remark holds for $T_R(M)\otimes_R Z$. Then we are done by (1), applied to $Y$.

For the ``only if" part of (3), we assume that $R$ is CM-free.  Gorenstein projective $T_R(M)$-modules are of the form $(X, u)$ with $u$ a monomorphism and ${\rm Cok} u\in R\mbox{-Gproj}$. Hence ${\rm Cok}u=P$ is projective. By Lemma \ref{lem:ind} and its proof, we infer that $(X, u)$ is isomorphic to $({\rm Ind}(P), c_P)$, which is identified with a projective $T_R(M)$-module.  For the ``if" part, we take any Gorenstein projective $R$-module $G$. Then $T_R(M)\otimes_R G$ is Gorenstein projective and thus projective. Then we are done by Lemma \ref{lem:pd}(1).
\end{proof}

\section{Perfect bimodules}

In this section, we study homological conditions, under which Theorem \ref{thm:2} applies. For this, we introduce the notion of a perfect bimodule; see Definiton \ref{defn:per}.

Let $R$ be a two-sided noetherian ring. We only consider finitely generated $R$-modules.  The following result is well known.

\begin{lem}\label{lem:class}
Let $Y$ be a right $R$-module. Suppose that we are given an exact sequence in $R\mbox{-{\rm mod}}$
$$\cdots \rightarrow E^{-n} \rightarrow E^{-(n-1)}\rightarrow \cdots \rightarrow E^{-1}\rightarrow E^0\rightarrow X\rightarrow 0$$
with ${\rm Tor}_i^R(Y, E^{-j})=0$ for each $i\geq 1$ and $j\geq 0$. Denote its $(-i)$-th cocycle   by $Z^{-i}$. Then the following statements hold.
\begin{enumerate}
\item The complex $Y\otimes_R E^\bullet$ computes ${\rm Tor}_i^R(Y, X)$, that is, $H^{-i}(Y\otimes_R E^\bullet)\simeq {\rm Tor}_i^R(Y, X)$ for $i\geq 0$.
\item There is an isomorphism ${\rm Tor}^R_i(Y, Z^{-j})={\rm Tor}^R_{i+j+1}(Y, X)$ for each $i\geq 1$ and $j\geq 0$.
    \item Assume that ${\rm pd}_{R^{\rm op}} Y<\infty$ and $F^\bullet$ is an acyclic complex with ${\rm Tor}_i^R(Y, F^{j})=0$ for each $i\geq 1$ and $j\in \mathbb{Z}$. Then the complex $Y\otimes_R F^\bullet$ is acyclic.
\end{enumerate}
\end{lem}

\begin{proof}
The statements (1) and (2) are classical. The last one follows from (1) and (2).
\end{proof}

The following consideration is related to the one in \cite{MY}. Let $M$ be an $R$-bimodule, which is finitely generated on both sides. We will consider the following Tor-vanishing conditions:

 \noindent{\rm  (P)} \quad  ${\rm Tor}^R_i(M, M^{\otimes_R j})=0$ for all $i,j\geq 1$.

\begin{lem}\label{lem:tensor}
We assume that the $R$-bimodule $M$ satisfies condition (P). Then the following statements are equivalent for each left $R$-module $Y$:
\begin{enumerate}
\item ${\rm Tor}_i^R(M, M^{\otimes_R j}\otimes_R Y)=0$ for any $i\geq 1$ and $j\geq 0$;
\item ${\rm Tor}_i^R(M^{\otimes_R s}, M^{\otimes_R j}\otimes_R Y)=0$ for any $i, s\geq 1$ and $j\geq 0$;
\item ${\rm Tor}_i^R(M^{\otimes_R s}, Y)=0$ for any $i, s\geq 1$.
\end{enumerate}
\end{lem}

We will call an $R$-module $Y$ \emph{$M$-flat}, provided that it satisfies one of the above equivalent conditions.

\begin{proof}
To show ``(1) $\Rightarrow$ (2)", we fix $j\geq 0$ and $s\geq 2$. Take a projective resolution $P^\bullet$ of $M^{\otimes_R j}\otimes_R Y$. It follows from (1) that $M\otimes_R P^\bullet$ is quasi-isomorphic to $M^{\otimes_R(j+1)}\otimes_R Y$. By Lemma \ref{lem:class}(1), the complex $M\otimes_R (M\otimes_R P^\bullet)$ is quasi-isomorphic to $M^{\otimes_R(j+2)}\otimes_R Y$, using the condition ${\rm Tor}_i^R(M, M^{\otimes_R(j+1)}\otimes_R Y)=0$ for each $i\geq 1$. Iterating this argument, we infer that $M^{\otimes_R s}\otimes_R P^\bullet$ is quasi-isomorphic to $M^{\otimes_R(j+s)}\otimes_R Y$. This proves (2). The implications  ``(2) $\Rightarrow$ (1)" and ``(2) $\Rightarrow$ (3)" are clear.

It remains to show ``(3) $\Rightarrow$ (1)". For this, we fix $j\geq 1$. Take a projective resolution $P^\bullet$ of $Y$. Then by (3), the complex $M^{\otimes_R j}\otimes_R P^\bullet$ is quasi-isomorphic to $M^{\otimes_R j}\otimes_R Y$. Applying Lemma \ref{lem:class}(1), we infer that $M \otimes_R (M^{\otimes_R j}\otimes_R P^\bullet) $ computes ${\rm Tor}_i^R(M,  M^{\otimes_R j}\otimes_R Y)$. But the complex is isomorphic to $M^{\otimes_R (j+1)}\otimes_R P^\bullet$, which is quasi-isomorphic to $M^{\otimes_R(j+1)}\otimes_R Y$ by (3). It follows that ${\rm Tor}_i^R(M,  M^{\otimes_R j}\otimes_R Y)=0$ for $i\geq 1$.
\end{proof}

The following consequence implies that condition (P) is symmetric.

\begin{cor}
Let $M$ be an $R$-bimodule. Then $M$ satisfies condition (P) if and only if ${\rm Tor}^R_i(M^{\otimes_R s}, M^{\otimes_R j})=0$ for each $i, s, j\geq 1$, if and only if ${\rm Tor}^R_i(M^{\otimes_R s}, M)=0$ for each $i, s\geq 1$.
\end{cor}

\begin{proof}
We just take $Y=M$ in the previous lemma.
\end{proof}

\begin{defn}\label{defn:per}
We call an $R$-bimodule $M$ \emph{perfect}, provided that it satisfies ${\rm pd}_R M<\infty$, ${\rm pd}_{R^{\rm op}} M<\infty$, and  condition (P). \hfill $\square$
\end{defn}

\begin{lem}\label{lem:per1}
Let $M$ be a perfect bimodule, and $Y\in R\mbox{-{\rm mod}}$. Then the following statements hold.
\begin{enumerate}
\item If ${\rm Tor}^R_i(M, Y)=0$ for each $i\geq 1$, then we have ${\rm pd}_R (M\otimes_R Y)\leq {\rm pd}_R M+{\rm pd}_RY$.
\item For each $i\geq 0$, we have ${\rm pd}_R M^{\otimes_R i}\leq i {\rm pd}_R M$ and ${\rm pd}_{R^{\rm op}} M^{\otimes_R i}\leq i {\rm pd}_{R^{\rm op}}M$.
    \item Assume  further that $M^{\otimes_R(N+1)}=0$ for $N\geq 1$. Then ${\rm pd}_R T_R(M)\leq N{\rm pd}_RM$ and ${\rm pd}_{R^{\rm op}} T_R(M)\leq N{\rm pd}_{R^{\rm op}}M$.
\end{enumerate}
\end{lem}

\begin{proof}
For (1), we assume that ${\rm pd}_R Y=n<\infty$. Take a projective resolution $P^\bullet$ of $Y$, which has length $n$. Then by assumption, the complex $M\otimes_R P^\bullet$ is quasi-isomorphic to $M\otimes_R Y$. We observe that each component of $M\otimes_R P^\bullet$ has projective dimension at most ${\rm pd}_R M$. Then (1) follows immediately. (2) follows from (1) by induction, and (3) follows from (2).
\end{proof}

We observe that Theorem \ref{thm:2} applies to nilpotent perfect bimodules.

\begin{prop}
Let $R$ be a two-sided noetherian ring, and $M$ be a nilpotent perfect $R$-bimodule. Then the conditions in Theorem \ref{thm:2} are fulfilled.
\end{prop}

\begin{proof}
Let $G$ be a Gorenstein projective $R$-module. Recall that ${\rm Ext}_R^i(G, X)=0={\rm Tor}_j^R(Y, G)$ for each $i, j\geq 1$, provided that ${\rm pd}_R X<\infty$ and ${\rm pd}_{R^{\rm op}} Y<\infty$.  Hence, by Lemma \ref{lem:per1}(2) we have ${\rm Ext}_R^i(G, M^{\otimes_R i})=0$ for $i\geq 0$. Moreover, it follows that $G$ is $M$-flat. Hence, by Lemma \ref{lem:tensor}(2) $M^{\otimes_R i}\otimes_R G$ is $M$-flat for $i\geq 1$. In particular, we have ${\rm Tor}_1^R(M, M^{\otimes_R i}\otimes_R G)=0$. Hence, $M$ is left-admissible. Similarly, it is right-admissible. The last condition follows from Lemma \ref{lem:class}(3), since each projective $T_R(M)$-module, as an $R$-module, is $M$-flat.
\end{proof}

\begin{thm}\label{thm:3}
Let $R$ be a two-sided noetherian ring, and $M$ be a nilpotent perfect $R$-bimodule. Then $R$ is Gorenstein if and only if so is $T_R(M)$. In this case, we have the following equalities
$${\rm G.dim}\;R-\delta \leq {\rm G.dim}\; T_R(M)\leq {\rm G.dim}\; R+\delta+1, $$
where $\delta={\rm min}\{{\rm pd}_{R}T_R(M), {\rm pd}_{R^{\rm op}}T_R(M)\}$.
\end{thm}

\begin{proof}
Set $l={\rm pd}_{R}T_R(M)$ and $r={\rm pd}_{R^{\rm op}}T_R(M)$. For the ``only if" part, we assume that $R$ is $d$-Gorenstein with $d={\rm G.dim}\; R$. We claim that ${\rm Gpd}_{T_R(M)}X\leq r+d+1$ for any left $T_R(M)$-module $X$.

For the claim, we take an exact sequence in $T_R(M)\mbox{-mod}$
\begin{align}\label{equ:sequence}
0\rightarrow Y\rightarrow Q^{1-r}\rightarrow \cdots \rightarrow Q^{-1}\rightarrow Q^0\rightarrow X\rightarrow 0
\end{align}
with each $Q^{-i}$ projective. In case that $r=0$, we set $Y=X$. Since $M$ is perfect, we infer that ${\rm Tor}_j^R(T_R(M), Q^{-i})=0$ for $j\geq 1$ and $0\leq i< r$. By Lemma \ref{lem:class}(2), we infer that ${\rm Tor}_j^R(T_R(M), Y)=0$ for $j\geq 1$. Hence, as $R$-modules,  $Y$ and thus $M\otimes_R Y$ are $M$-flat. Recall the isomorphism (\ref{equ:iso-cat}). Then by (\ref{equ:exact}), we have an exact sequence in $T_R(M)\mbox{-mod}$
\begin{align}\label{equ:sequence2}0\longrightarrow T_R(M)\otimes_R(M\otimes_R Y)\longrightarrow T_R(M)\otimes_R Y\longrightarrow Y\longrightarrow 0.\end{align}
Since ${\rm Gpd}_RY\leq d$ and ${\rm Gpd}_R M\otimes_R Y\leq d$, we infer from Theorem \ref{thm:2}(2) that ${\rm Gpd}_{T_R(M)} T_R(M)\otimes_R(M\otimes_R Y)\leq d$ and ${\rm Gpd}_{T_R(M)}  T_R(M)\otimes_R Y\leq d$. Therefore, the exact sequence implies that ${\rm Gpd}_{T_R(M)} Y\leq d+1$, which implies the claim.

Similarly, we prove that ${\rm Gpd}_{T_R(M)^{\rm op}}X\leq l+d+1$ for any right $T_R(M)$-module $X$.  Then we are done by Lemma \ref{lem:goren}(2).

For the ``if" part, we assume that $T_R(M)$ is $d'$-Gorenstein. For any left $R$-module $Z$, we consider the exact sequence of $R$-modules
$$0\rightarrow K\rightarrow P^{1-r}\rightarrow \cdots \rightarrow P^{-1}\rightarrow P^0\rightarrow X\rightarrow 0$$
with each $P^{-i}$ projective. If $r=0$, we set $K=Z$. By a dimension-shift, we infer that ${\rm Tor}_j^R(T_R(M), K)=0$ for each $j\geq 1$. Hence $K$ is $M$-flat. By Theorem \ref{thm:2}(2), we have ${\rm Gpd}_R K={\rm Gpd}_{T_R(M)} T_R(M)\otimes_R K\leq d'$. Hence, we have ${\rm Gpd}_R Z\leq r+d'$. By a similar statement for right $R$-modules and Lemma \ref{lem:goren}(2), we complete the proof.
\end{proof}

The global dimension of a ring $S$ is denoted by ${\rm gl.dim}\; S$.

\begin{cor}\label{cor:1}
Keep the same assumptions as in Theorem \ref{thm:3}. Then $R$ has finite global dimension if and only if so does $T_R(M)$, in which case we have
$${\rm gl.dim}\;R-\delta \leq {\rm gl.dim}\; T_R(M)\leq {\rm gl.dim}\; R+\delta+1. $$
\end{cor}

\begin{proof}
Recall that a  two-sided noetherian ring $S$ has finite global dimension if and only if it is Gorenstein and CM-free, in which case we have ${\rm G.dim}\; S={\rm gl.dim}\; S$. Then we are done by Theorem \ref{thm:2}(3).
\end{proof}

\begin{lem}
Keep the same assumptions as in Theorem \ref{thm:3}. Let $X$ be a left $T_R(M)$-module. Then ${\rm pd}_{T_R(M)} X<\infty$ if and only if ${\rm pd}_R X<\infty$.
\end{lem}

\begin{proof}
The ``only if" part is trivial, since ${\rm pd}_R T_R(M)<\infty$. For the ``if" part, we assume that ${\rm pd}_RX<\infty$. Take the exact sequence (\ref{equ:sequence}) as above. Hence, as $R$-modules, $Y$ and $M\otimes_R Y$ are $M$-flat. We observe that ${\rm pd}_R Y<\infty$ and by Lemma~\ref{lem:per1}(1) ${\rm pd}_R (M\otimes_R Y)<\infty$. Hence, we infer by Lemma \ref{lem:pd}(2) that the left two terms in (\ref{equ:sequence2}) have finite projective dimension.  Then so does $Y$. We infer from (\ref{equ:sequence}) that the $T_R(M)$-module $X$ has finite projective dimension.
\end{proof}

Let $k$ be a field. We denote by $D={\rm Hom}_k(-, k)$ the $k$-dual. For quivers, we refer to \cite{ARS}.

\begin{exm}
{\rm (1) Let $M$ be a nilpotent $R$-bimodule, which is finitely generated projective on both sides. Then $M$ is perfect. Indeed, the tensor ring $T_R(M)$ is a projective left and right $R$-module. It follows from Theorem \ref{thm:3} that if $R$ is $d$-Gorenstein, then $T_R(M)$ is $(d+1)$-Gorenstein.  Taking $d=0$, we recover the examples in \cite[Section 6]{GLS} and \cite{Wang, LY}.

Let $R$ be the $k$-algebra given by the following quiver
\[\xymatrix{1 \ar[r]_{\alpha_1} & 2 \ar[r]_{\alpha_2} & 3\ar @/_0.9pc/[ll]_{\alpha_3}}\]
subject to relations $\{\alpha_2\alpha_1, \alpha_3\alpha_2, \alpha_1\alpha_3\}$; it is a quasi-Frobenius algebra. We denote by $e_i$ the idempotent corresponding to the vertex $i$. Take $M=Re_1\otimes_k e_3R$, which is an $R$-bimodule, projective on both sides. We observe that $M^{\otimes_R 2}=0$. It follows that $T_R(M)=R\ltimes M$,  the \emph{trivial extension} of $R$ by $M$. By Theorem \ref{thm:3}, we infer that $T_R(M)$ is $1$-Gorenstein. Indeed, the algebra $T_R(M)$ is a gentle algebra.

\vskip 3pt

(2) Let $R_1$ and $R_2$ be two-sided noetherian rings, and  $M$ be an $R_1$-$R_2$-bimodule which is finitely generated on both sides. Set $R=R_1\times R_2$. Then $M$ becomes naturally an $R$-module. Then the tensor ring $T_R(M)$ is isomorphic to the formal matrix ring $\begin{pmatrix}R_1 &  M\\ 0 & R_2\end{pmatrix}$. We observe that $M$ is a perfect $R$-module if and only if $M$ has finite projective dimension as a left $R_1$-module and a right $R_2$-module. Then Theorem \ref{thm:3} recovers \cite[Theorem 2.2(iii)]{XZ}; compare \cite[Theorem 3.3]{Chen12}. It is of interest to compare the conditions in Theorem \ref{thm:2} with the compatible bimodule in \cite[Definition 1.1]{Zhang}.

\vskip 3pt

(3)  This is a non-example of Corollary \ref{cor:1}.  Let $R$ be a tilted $k$-algebra of a finite acyclic quiver $Q$, in particular, its global dimension is at most two. Denote by $M={\rm Ext}^2_R(DR, R)$,  which is naturally an $R$-bimodule. We observe that $M^{\otimes_R 2}=0$; compare \cite[Proposition 4.7]{Am}. It follows that $T_R(M)=R\ltimes M$, which is isomorphic to the cluster-tilted algebra of $Q$; see \cite{ABS}. Recall from \cite{KR} that a cluster-tilted algebra  is $1$-Gorenstein, which usually has infinite global dimension. Hence, the bimodule $M$ is not perfect in general. We mention that  Gorenstein homological properties of trivial extensions are studied in \cite{MY}.  }
\end{exm}

\vskip 5pt

\noindent {\bf Acknowledgements.} \quad The project was initiated when both authors was visiting  the University of Bielefeld. We thank Henning Krause for his hospitality. The authors are supported by the National Natural Science Foundation of China (Nos. 11522113, 11671245, 11401401 and 11601441).

\bibliography{}

\vskip 10pt

 {\footnotesize \noindent Xiao-Wu Chen\\
 Key Laboratory of Wu Wen-Tsun Mathematics, Chinese Academy of Sciences,\\
 School of Mathematical Sciences, University of Science and Technology of China, Hefei 230026, Anhui, PR China}

 \vskip 5pt

 {\footnotesize \noindent Ming Lu\\
  Department of Mathematics, Sichuan University, Chengdu 610064, PR China }

\end{document}